\documentclass{amsart}
\usepackage{amssymb}
\usepackage{amsfonts}
\usepackage{fancyhdr}

\newtheorem{theorem}{Theorem}
\theoremstyle{plain}

\newtheorem{claim}{Claim}

\newtheorem{corollary}{Corollary}

\newtheorem{definition}{Definition}

\newtheorem{lemma}{Lemma}

\newtheorem{proposition}{Proposition}
\newtheorem{remark}{Remark}

\numberwithin{equation}{section}
\input{tcilatex}
\begin{document}
\title{Groups equal to a product of three conjugate subgroups}
\author{John Cannon}
\address[John Cannon]{Computational Algebra Group \\
School of Mathematics and Statistics\\
The University of Sydney\\
Sydney, NSW 2006\\
Australia}
\email{john@maths.usyd.edu.au}
\author{Martino Garonzi}
\address[Martino Garonzi]{Department of Mathematics\\
University of Padova\\
Via Trieste 63\\
35121 Padova\\
Italy}
\email{mgaronzi@gmail.com}
\author{Dan Levy}
\address[Dan Levy]{The School of Computer Sciences \\
The Academic College of Tel-Aviv-Yaffo \\
2 Rabenu Yeruham St.\\
Tel-Aviv 61083\\
Israel}
\email{danlevy@mta.ac.il}
\author{Attila Mar\'{o}ti}
\address[Attila Mar\'{o}ti]{Fachbereich Mathematik, Technische Universit\"{a}%
t Kaiserslautern\\
Postfach 3049, 67653 Kaiserslautern\\
Germany \\
and Alfr\'{e}d R\'{e}nyi Institute of Mathematics\\
Re\'{a}ltanoda utca 13-15\\
H-1053, Budapest\\
Hungary}
\email{maroti@mathematik.uni-kl.de and maroti.attila@renyi.mta.hu}
\thanks{A.M. acknowledges the support of an Alexander von Humboldt
Fellowship for Experienced Researchers and the support of OTKA K84233.}
\author{Iulian I. Simion}
\address[Iulian I. Simion]{Department of Mathematics\\
University of Padova\\
Via Trieste 63\\
35121 Padova\\
Italy}
\email{iulian.simion@math.unipd.it}
\thanks{I.S. acknowledges the support of the Swiss National Science
Foundation (project number P2ELP2\_148913) as well as the support of the
University of Padova by grants CPDR131579/13 and CPDA125818/12 . }
\date{\today }
\subjclass[2000]{ 20D40, 20E42, 20D05, 20B40}
\keywords{double cosets, triple factorization, $BN$-pairs, dioid}

\begin{abstract}
Let $G$ be a finite non-solvable group. We prove that there exists a proper
subgroup $A$ of $G$ such that $G$ is the product of three conjugates of $A$,
thus replacing an earlier upper bound of $36$ with the smallest possible
value. The proof relies on an equivalent formulation in terms of double
cosets, and uses the following theorem which is of independent interest and
wider scope: Any group $G$ with a $BN$-pair and a finite Weyl group $W$
satisfies $G=\left( Bn_{0}B\right) ^{2}=BB^{n_{0}}B$ where $n_{0}$ is any
preimage of the longest element of $W$. The proof of the last theorem is
formulated in the dioid consisting of all unions of double cosets of $B$ in $%
G$. Other results on minimal length product covers of a group by conjugates
of a proper subgroup are given.
\end{abstract}

\maketitle

\section{ Introduction\label{Sect_Intro}}

This paper continues and expands the study, initiated in \cite{GaronziLevy},
of minimal length factorizations of groups in terms of conjugates of a
proper subgroup. Here we do not limit the discussion to finite groups and
therefore we define $\gamma _{\text{cp}}\left( G\right) $, for any group $G$%
, to be the minimal integer $k$ such that $G$ is a product of $k$ conjugates
of a proper subgroup of $G$, and $\gamma _{\text{cp}}\left( G\right) =\infty 
$ if no such $k$ exists. Note that any $G$ such that every maximal subgroup
of $G$ is normal, has $\gamma _{\text{cp}}\left( G\right) =\infty $, so in
particular, if $G$ is finite, $\gamma _{\text{cp}}\left( G\right) =\infty $
if and only if $G$ is nilpotent. In \cite{GaronziLevy} it was shown that $%
\gamma _{\text{cp}}\left( G\right) $ cannot be bounded by a universal
constant in the class of finite solvable non-nilpotent groups while $\gamma
_{\text{cp}}\left( G\right) \leq 36$ for any finite non-solvable $G$. In the
present paper we prove that $\gamma _{\text{cp}}\left( G\right) =3$ for any
finite non-solvable $G$ (Theorem \ref{Th_gamma=3} below). Note that for any
non-trivial $G$ we have $\gamma _{\text{cp}}\left( G\right) \geq 3$.

Once the problem is reduced to the almost simple case, the proof of $\gamma
_{\text{cp}}\left( G\right) \leq 36$ uses two related ingredients which we
would like to recall. The first one is the interpretation of the condition $%
G=$ $A_{1}\cdots A_{k}$, where $A_{i}$ is a conjugate of $A<G$ for any $%
1\leq i\leq k$, in terms of the right multiplication action of $G$ on $%
\Omega _{A}:=\left\{ Ax|x\in G\right\} $ (which induces a right
multiplication action of any subgroup of $G$ on $\Omega _{A}$). The second
one is a relation between products of conjugates of $A$ and products of
double cosets of $A$. Our first result can be viewed as an extension of
these observations. We denote $x^{-1}Ax$ by $A^{x}$ for any $A\leq G$ and $%
x\in G$, and for any two sets $A$ and $B$, "$A$ intersects $B$" is
equivalent to $A\cap B\neq \phi $.

\begin{theorem}
\label{Th_ConjDCEquiv}Let $G$ be a group and $A\leq G$.

1. Let $x,y\in G$. Then $G=AA^{x}A^{y}$ implies $y\in AA^{x}$. Moreover,
assuming that $y\in AA^{x}$, the following three conditions are equivalent:

(a) $G=AA^{x}A^{y}$.

(b) $G=\left( AzA\right) \left( AwA\right) $ where $z=x^{-1}$ and $w=xy^{-1}$%
.

(c) $G=AA^{x}A=\left( Ax^{-1}A\right) \left( AxA\right) $.

2. Let $x\in G$ then:

(i) $G=AA^{x}A$ if and only if $Ox$ intersects every orbit of the action of $%
A$ on $\Omega _{A}$, where $O$ is the unique $A$-orbit satisfying $%
Ax^{-1}\in O$. Furthermore, $Ox$ intersects every orbit of the action of $A$
on $\Omega _{A}$ if and only if $Oz$ intersects every orbit of the action of 
$A$ on $\Omega _{A}$, where $z\in AxA$ is arbitrary.

(ii) $G=\left( AxA\right) ^{2}\Longleftrightarrow G=AA^{x}A^{x^{2}}$.
\end{theorem}

\begin{theorem}
\label{Th_gamma=3}Let $G$ be a finite non-solvable group. Then there exists $%
A<G$ and $x\in G$ such that $G=\left( AxA\right) ^{2}$. In particular $%
\gamma_{\text{cp}}\left( G\right) =3$.
\end{theorem}

\begin{remark}
We do not know if for a general group $G$ the existence of $z,w\in G$ such
that $G=\left( AzA\right) \left( AwA\right) $ implies the existence of $x\in
G$ such that $G=\left( AxA\right) ^{2}$, or even if the seemingly weaker
implication which states that $\gamma _{\text{cp}}\left( G\right) =3$
implies the existence of $A<G$ and $x\in G$ such that $G=\left( AxA\right)
^{2}$, is true.
\end{remark}

\begin{corollary}
\label{Coro_AtMost3}Any finite group which is not a cyclic $p$-group has a
factorization of the form $G=ABA$, where $A$ and $B$ are proper subgroups of 
$G$.
\end{corollary}

A major part of the proof of Theorem \ref{Th_gamma=3} relies on the next
theorem, which is, however, of wider scope and interest. We recall that a
group $G$ is said to be a group with a $BN$-pair (\cite{Carter2}, Section
2.1) if there exist subgroups $B$ and $N$ of $G$ such that: (i) $%
G=\left\langle B,N\right\rangle $; (ii) $H:=B\cap N\trianglelefteq N$; (iii) 
$W:=N/H$ (the Weyl group of the $BN$-pair) is generated by a set of elements 
$s_{i}$, $i\in I$, where $I$ is some indexing set, and $s_{i}^{2}=1$ for all 
$i\in I$; (iv) For any $n_{i}\in N$ such that $s_{i}=n_{i}H$, it holds that $%
n_{i}Bn_{i}\neq B$; (v) $n_{i}Bn\subseteq Bn_{i}nB\cup BnB$ for all $n\in N$
and $n_{i}\in N$ such that $s_{i}=n_{i}H$. By assumption, any element of $W$
can be written as a product of some of the $s_{i}$. We define the length
function $l:W\rightarrow \mathbb{N}_{0}$ (\cite{Carter2}, Section 2.1), by
the conditions $l\left( 1_{W}\right) =0$ and for any $1_{W}\neq w\in W$, the
positive integer $l\left( w\right) $ is the minimal length of an expression
for $w$ as a product of the $s_{i}$. If $W$ is finite there exists a
(unique) element $w_{0}\in W$ such that $l\left( w_{0}\right) >l\left(
w\right) $ for all $w\in W$, $w\neq w_{0}$ (\cite{Carter2}, Proposition
2.2.11).

\begin{theorem}
\label{Th_G=BB_B for BN-pair}Let $G$ be a group with a $BN$-pair and a
finite Weyl group. Let $n_{0}\in N$ be such that $n_{0}H=w_{0}$. Then $%
G=\left( Bn_{0}B\right) ^{2}=BB^{n_{0}}B$.
\end{theorem}

\begin{remark}
\label{Rem_InfiniteWeyl}If $W$ is infinite then $G$ is not equal to any
(finite) product of double cosets of $B$.
\end{remark}

\begin{corollary}
\label{Coro_Gauss}Let $G$ be a group with a $BN$-pair and a finite Weyl
group $W$. Let $U\leq B$ satisfy $B=UH$. Then $G=HUU^{n_{0}}U$.
\end{corollary}

Note that in particular, Corollary \ref{Coro_Gauss} applies to any group
with a split $BN$-pair (\cite{Carter2}, Section 2.5). For the origin,
applications and discussion of necessary and sufficient conditions for the
identity $G=HUU^{n_{0}}U$ in the context of Chevalley groups, we refer the
reader to \cite{VavilovSmolenskySury2}.

In the case of simple groups of Lie type, the proof of Theorem \ref%
{Th_G=BB_B for BN-pair} can be obtained from \cite{Kawanaka} where the
structure constants of Iwahori-Hecke algebras are described. The starting
point for the characterization of these constants is \cite{BorelTits}, which
describes the product of double cosets for Borel subgroups in the context of
connected reductive linear algebraic groups. Although in proving Theorem \ref%
{Th_gamma=3} we do make use of the Hecke algebra approach which represents
double cosets by group algebra sums, and then computes the structure
constants, the proof of Theorem \ref{Th_G=BB_B for BN-pair} takes advantage
of the fact that one only needs to distinguish between zero and non-zero
structure constants. This leads us to consider another algebraic structure,
which we call the DC dioid, to handle double cosets. Both approaches are
introduced and discussed in Section \ref{Sect_DC}.

A special case of Theorem \ref{Th_G=BB_B for BN-pair} is the case where $G$
is a connected reductive linear algebraic group (see for example [\cite%
{MalleTesterman}, Theorem 11.16]). In this case one can offer a different
proof which is based on a topological argument. This argument is applicable
to any group $G$ such that $G$ is a topological space, and right and left
multiplications by fixed elements of $G$ are continuous maps (such a group
is customarily called a semi-topological group). The core of the argument
relies on standard topological considerations, and it is clear and easy to
grasp, so we hope it can be applied for other groups as well.

\begin{proposition}
\label{Prop_semiTopo}Let $G$ be a semi-topological group. If $A$ is an open
dense subset of $G$ then $G=AA^{-1}$. In particular, if $A=A^{-1}$ then $%
G=A^{2}$.
\end{proposition}

\begin{corollary}
\label{Coro_GReducedG=BBBFromTopo}Let $G$ be a connected reductive linear
algebraic group, $B$ a Borel subgroup of $G$, and $n_{0}$ as in Theorem \ref%
{Th_G=BB_B for BN-pair}. Then, since $Bn_{0}B$ is open and dense in $G$, $%
G=\left( Bn_{0}B\right) ^{2}$.
\end{corollary}

Returning to the setting of Theorem \ref{Th_G=BB_B for BN-pair}, let $G$ be
a group with $BN$-pair and finite Weyl group $W$. For any subset $I^{\prime
}\subseteq I$, the standard parabolic subgroup $W_{I^{\prime }}$ of $W$ is
defined by $W_{I^{\prime }}:=\left\langle s_{i}|i\in I^{\prime
}\right\rangle $. There is a bijection (\cite{Carter1}, Section 8.3) between
overgroups of $B$ (the standard parabolic subgroups of $G$ w.r.t $B$ and $N$%
) and standard parabolic subgroups of $W$ given by $W_{I^{\prime
}}\longleftrightarrow BW_{I^{\prime }}B$ (here $W_{I^{\prime }}$ is regarded
as a union of left cosets of $H$ in $N$). Furthermore, since any standard
parabolic subgroup $P$ of $G$ contains $B$, it is immediate from Theorem \ref%
{Th_G=BB_B for BN-pair} that $G=PP^{n_{0}}P$. Since $W$ is a finite Coxeter
group (see \cite{Carter2} Chapter 2) it is natural to consider the question
of which irreducible finite Coxeter groups (the terminology is explained in
Section \ref{SECT_Coxeter}) are equal to a product of three conjugate
parabolic subgroups. This is answered by Theorem \ref{Th_WeylParabolics}.
Note that by the above mentioned bijection between the standard parabolics
of $W$ and those of $G$, to any \ factorization of $W$ by three conjugates
of some $W_{I^{\prime }}$ there is a naturally corresponding factorization
of $G$ (see Remark \ref{Rem_FromWToG} below).

\begin{theorem}
\label{Th_WeylParabolics}Let $C$ be a finite irreducible Coxeter group. Then 
$C$ is the product of three conjugates of at least one of its proper
parabolic subgroups if $C$ is of one of the following types: $A_{n}$, $%
n\geq2 $, $B_{n}$, $n\geq3$, $D_{n}$, $n\geq4$, $H_{4}$, $E_{6}$, $E_{7}$, $%
E_{8}$. Otherwise, for the remaining types $A_{1}$, $B_{2}$, $F_{4}$, $H_{3}$%
, $I_{2}\left( m>4\right) $, $C$ is not a product of three conjugates of a
proper parabolic subgroup.
\end{theorem}

\begin{remark}
\label{Rem_FromWToG}Let $G$ be a group with $BN$-pair and a finite Weyl
group $W$ whose generators are indexed by $I$. Suppose that $W$ is the
product of three conjugates of the standard parabolic subgroup $W_{I^{\prime
}}$ where $I^{\prime }\subseteq I$. Then $G$ is the product of the
corresponding three conjugates of the standard parabolic subgroup\ $%
P_{I^{\prime }}=BW_{I^{\prime }}B$.
\end{remark}

Finally, we mention other recent papers which discuss similar problems to
those addressed in the present paper. Liebeck, Nikolov and Shalev (\cite%
{LiebeckNikolovShalev1}) consider products of conjugate subgroups in finite
simple groups. They conjecture that there exists a universal constant $c$
such that for every non-trivial subgroup $A$ of a simple group $G$, the
minimal number of conjugates of $A$ such that the setwise product of these
conjugates is $G$, is bounded above by $c\log \left\vert G\right\vert /\log
\left\vert A\right\vert $. Later, in \cite{LiebeckNikolovShalev2}, they
extend this conjecture to subsets of $G$ of size at least $2$. For a
discussion of triple factorizations of the form $G=ABA$ for $A,B<G$, see 
\cite{AlaviBurness} and \cite{AlaviPrager} and the references therein.
Liebeck and Pyber have proved (\cite{LiebeckPyber} Theorem D) that every
finite simple group of Lie type is a product of no more than $25$ Sylow $p$%
-subgroups, where $p$ is the defining characteristic. In \cite%
{BabaiNikolovPyber} it is claimed, without proof, that the $25$ can be
replaced by $5$, while a sketch of a proof of this claim for exceptional Lie
type groups appears in a survey by Pyber and Szab\'{o} (\cite{PyberSzabo}
Theorem 15). For (non-twisted) Chevalley groups the best possible bound of $%
4 $ is shown to hold by Smolensky, Sury and Vavilov in \cite%
{VavilovSmolenskySury}.

\section{Products of Double Cosets\label{Sect_DC}}

In this section we prove Theorem \ref{Th_ConjDCEquiv} and describe two
algebraic frameworks for discussing products of double cosets.

Let $G$ be any group and $A\leq G$. Recall that the set of all double cosets
of $A$ in $G$ forms a partition of $G$ and $AxA=Ax^{\prime }A$ if and only
if $x^{\prime }\in AxA$. By a product of two double cosets of $A$ we mean
their setwise product, and $\left( AzA\right) \left( AwA\right) $ is equal
to a disjoint union of some double cosets of $A$.

The following lemma is needed for the proof of Theorem \ref{Th_ConjDCEquiv}
and, in fact, gives a more general version of part of its claims.

\begin{lemma}
\label{Lem_GeneralCOnjDCrelation}Let $G$ be a group, $A\leq G$ and $%
A_{1},A_{2},\ldots ,A_{k}$ are $k\geq 3$ conjugates of $A$ in $G$. Then $%
G=A_{1}A_{2}\cdots A_{k}$ if and only if $G=\left( Ax_{1}A\right) \left(
Ax_{2}A\right) \cdots \left( Ax_{k-1}A\right) $ for some $x_{1},\ldots
,x_{k-1}\in G$.
\end{lemma}

\begin{proof}
1. Suppose that $G=\left( Ax_{1}A\right) \left( Ax_{2}A\right) \cdots \left(
Ax_{k-1}A\right) $ for some $x_{1},\ldots ,x_{k-1}\in G$. Then%
\begin{align*}
G& =\left( Ax_{1}A\right) \left( Ax_{2}A\right) \cdots \left(
Ax_{k-1}A\right) =Ax_{1}Ax_{2}A\cdots Ax_{k-1}A= \\
& =AA^{x_{1}^{-1}}A^{\left( x_{1}x_{2}\right) ^{-1}}\cdots A^{\left(
x_{1}x_{2}\cdots x_{k-2}\right) ^{-1}}A^{\left( x_{1}x_{2}\cdots
x_{k-1}\right) ^{-1}}x_{1}x_{2}\cdots x_{k-1}\text{,}
\end{align*}%
and multiplying both sides on the right by $\left( x_{1}x_{2}\cdots
x_{k-1}\right) ^{-1}$, gives%
\begin{equation*}
G=AA^{x_{1}^{-1}}A^{\left( x_{1}x_{2}\right) ^{-1}}\cdots A^{\left(
x_{1}x_{2}\cdots x_{k-2}\right) ^{-1}}A^{\left( x_{1}x_{2}\cdots
x_{k-1}\right) ^{-1}}\text{.}
\end{equation*}%
2. Suppose that $G=A_{1}A_{2}\cdots A_{k}$ where each $A_{i}$ is a conjugate
of $A$ in $G$ and $k\geq 3$. Let $g_{i}\in G$, $1\leq i\leq k$ be such that $%
A_{i}=A^{g_{i}}$. Then:%
\begin{equation*}
G=A_{1}A_{2}\cdots A_{k}=\left( g_{1}^{-1}Ag_{1}\right) \left(
g_{2}^{-1}Ag_{2}\right) g_{3}^{-1}A\cdots g_{k-1}\left(
g_{k}^{-1}Ag_{k}\right) \text{.}
\end{equation*}%
Multiplying by $g_{1}$ on the left and $g_{k}^{-1}$ on the right, gives:%
\begin{align*}
G& =Ag_{1}g_{2}^{-1}Ag_{2}g_{3}^{-1}A\cdots
g_{k-1}g_{k}^{-1}A=Ag_{1}g_{2}^{-1}AAg_{2}g_{3}^{-1}AA\cdots
AAg_{k-1}g_{k}^{-1}A= \\
& =\left( Ag_{1}g_{2}^{-1}A\right) \left( Ag_{2}g_{3}^{-1}A\right) \cdots
\left( Ag_{k-1}g_{k}^{-1}A\right) \text{.}
\end{align*}
\end{proof}

\begin{proof}[Proof of Theorem \protect\ref{Th_ConjDCEquiv}]
1. First we show that (a) $\Longrightarrow $ (b) independent of the
assumption $y\in AA^{x}$: Take $k=3$ in Lemma \ref{Lem_GeneralCOnjDCrelation}%
, and in part (2) of its proof, $A_{1}=A$, $g_{1}=1$, $g_{2}=x$ and $g_{3}=y$%
. We get (b) for $z=x^{-1}$ and $w=xy^{-1}$. Now let $z,w$ be any two
elements of $G$ such that $G=\left( AzA\right) \left( AwA\right) $. Then, in
particular, $1_{G}\in \left( AzA\right) \left( AwA\right) $, from which it
follows that there exist $a_{1},a_{2}\in A$ such that $\left(
a_{1}wa_{2}\right) ^{-1}=a_{2}^{-1}w^{-1}a_{1}^{-1}\in AzA$. It follows that 
$w^{-1}\in AzA$, and so there exist $a_{3},a_{4}\in A$ such that $%
w^{-1}=a_{3}za_{4}$. Therefore, since (a) $\Longrightarrow $ (b), we get
that (a) implies $yx^{-1}=a_{3}x^{-1}a_{4}$ which is equivalent to $%
y=a_{3}x^{-1}a_{4}x\in AA^{x}$. Moreover, since (b) implies $w^{-1}\in AzA$
and hence $AwA=Az^{-1}A$, substituting $z=x^{-1}$ in (b) gives (c), so (b)
implies (c). Finally we prove (c) $\Longrightarrow $ (a) for any $y\in
AA^{x} $. Assume $G=AA^{x}A$, and let $y=a_{3}x^{-1}a_{4}x$, where $%
a_{3},a_{4}\in A $ are arbitrary. Then 
\begin{eqnarray*}
AA^{x}A^{y} &=&AA^{x}A^{a_{3}x^{-1}a_{4}x}=Ax^{-1}Ax\left(
x^{-1}a_{4}x\right) ^{-1}Ax^{-1}a_{4}x= \\
&=&Ax^{-1}AxAx^{-1}a_{4}x=AA^{x}A\left( x^{-1}a_{4}x\right) =G\text{.}
\end{eqnarray*}

2. (i) We have $O=\left( Ax^{-1}\right) A$ and therefore $AA^{x}=Ox$, and
hence, since $G$ is the union of all right cosets of $A$, the condition $%
G=\left( AA^{x}\right) A$ is equivalent to $\left( Ox\right) A=\Omega _{A}$.
The last equality is equivalent to the statement that $Ox$ intersects every
orbit of the action of $A$ on $\Omega_{A}$, since $\Omega_{A}$ is the union
of all of the orbits of $A$ on $\Omega_{A}$. Finally, if $z\in AxA$ then $%
z=a_{1}xa_{2}$, with $a_{1},a_{2}\in A$. So $Oz=\left( Ax^{-1}\right)
Aa_{1}xa_{2}=\left( Ax^{-1}\right) Axa_{2}=\left( Ox\right) a_{2}$, and $Ox$
intersects every orbit of the action of $A$ on $\Omega_{A}$ if and only if $%
\left( Ox\right) a_{2}$ does.

(ii) First note (by taking inverses) that $G=\left( AxA\right)
^{2}\Longleftrightarrow G=\left( Ax^{-1}A\right) ^{2}$. Now take $k=3$ and $%
x_{1}=x_{2}=x^{-1}$ in the proof of Lemma \ref{Lem_GeneralCOnjDCrelation}
part 1. Thus $G=\left( AxA\right) ^{2}$ implies $G=AA^{x}A^{x^{2}}$.
Conversely, if $G=AA^{x}A^{x^{2}}$, we take $k=3$, $g_{1}=1$, $g_{2}=x$ and $%
g_{3}=x^{2}$ in the proof part 2 of Lemma \ref{Lem_GeneralCOnjDCrelation}\
and use the above mentioned equivalence to get\ $G=\left( AxA\right) ^{2}$.
\end{proof}

\bigskip

Let $G$ be a group and $A<G$. We discuss two possible algebraic frameworks
to deal with products of double cosets of $A$. Fix a set $J\subseteq G$ of
representatives of distinct double cosets of $A$ in $G$.

\subsection{The Hecke algebra of double cosets}

For the first algebraic framework, it is sufficient, for our purpose, to
assume that $G$ is finite. Let $\mathbb{Q}$ be the field of rational numbers
and $\mathbb{Q}\left[ G\right] $ the group algebra of $G$ over $\mathbb{Q}$.
For any subset $S\subseteq G$ define $\underline{S}\in \mathbb{Q}\left[ G%
\right] $ by $\underline{S}:=\tsum\limits_{g\in S}$ $g$. The set $\left\{
e_{j}:=\frac{1}{\left\vert A\right\vert }\underline{AjA}|j\in J\right\} $ is
linearly independent in $\mathbb{Q}\left[ G\right] $, and its elements
satisfy the product rule $e_{x}e_{y}=\tsum\limits_{j\in J}a_{xyj}e_{j}$,
where the structure constants $a_{xyj}$ (also called \emph{intersection
numbers}) are nonnegative integers. The span of the set $\left\{ e_{j}:=%
\frac{1}{\left\vert A\right\vert }\underline{AjA}|j\in J\right\} $ in $%
\mathbb{Q}\left[ G\right] $\ is the Hecke algebra of the double cosets of $A$
- see \cite{CurtisReiner} Chapter 1, Section 11D. The advantage in
representing the double cosets as group algebra sums lies in its relation to
the action of $G$ on $\Omega _{A}$. This relation gives a highly non-trivial
algorithm to compute the $a_{xyj}$ from the permutation character $1_{A}^{G}$
associated with the action of $G$ on $\Omega _{A}$, viewed as a complex
character - see \cite{Cameron1} Chapters 2 and 3 for some general background
and \cite{BreuerLux}, \cite{Muller} for the details of the specific
algorithm we later use in Section \ref{Sect_gamma=3}. This algorithm
requires the additional assumption that the character $1_{A}^{G}$ is
multiplicity free. Since we are assuming that $G$ is finite, the number $%
r:=\left\vert J\right\vert $ of distinct double cosets of $A$ is a natural
number that is equal to the rank of the permutation representation of $G$ on 
$\Omega _{A}$ (see Exercise 3.2.27 of \cite{DixonUndMortimer}). It is
customary to view the $a_{xyj}$ as defining $r$ square nonnegative integer $%
r\times r$ matrices called the \emph{collapsed adjacency matrices}, via $%
\left( P_{y}\right) _{xj}=a_{xyj}$. Typically, the rank of the given
permutation representation is much smaller than its degree $\left\vert
G:A\right\vert $, and this explains why the above mentioned algorithm is
capable of computing the collapsed adjacency matrices even for the larger
sporadic groups and some of their subgroups. Note that for a given $G$ and $%
A $ and $x,y\in J$, the condition $G=\left( AxA\right) \left( AyA\right) $
is equivalent to the condition $\left( P_{y}\right) _{xj}\neq 0$ for all $%
j\in J $, namely, that the row labeled by $x$ in the matrix labeled by $y$
consists entirely of non-zero entries.

\subsection{The dioid of double cosets}

Our second approach to products of double cosets of $A$ in $G$ (here $G$ can
be any group, not necessarily finite) utilizes the following algebraic
structure.

\begin{definition}
\label{Def_semiring}A quintuple $\left( R,+,\cdot,0,1\right) $ where $R$ is
a set and $+$ and $\cdot$ are two binary operations over $R$, called
respectively addition and multiplication, is a \emph{semiring} if the
following axioms are satisfied:

(a) $\left( R,+,0\right) $ is a commutative monoid with an identity element $%
0$.

(b) $\left( R,\cdot,1\right) $ is a monoid with an identity element $1$.

(c) Multiplication is right and left distributive over addition.

(d) Multiplication by $0$ annihilates $R$, that is $0\cdot a=a\cdot0=0$ for
all $a\in R$.

A \emph{starred semiring} is a semiring equipped with an additional unary
operation denoted by $\ast $ (no axioms are imposed in the general case). A
semiring in which addition is idempotent, namely, $a+a=a$ for all $a\in R$,
is called an \emph{idempotent semiring} or a \emph{dioid} (\cite{Gunawardena}%
). Note that every dioid is equipped with a "ready made" partial order
relation defined by $a\leq b$ if and only if $a+b=b$. A \emph{complete
semiring} is a semiring with an infinitary sum operation, namely, $%
\tsum\limits_{i\in I}a_{i}$ is defined for any indexing set $I$ with $%
a_{i}\in R$ for every $i\in I$. An infinitary sum reduces to an ordinary
finite sum when $I$ is finite and behaves in a natural way with respect to
partitions of $I$. Moreover, it is required to satisfy the left and right
distributive laws: $a\cdot \left( \tsum\limits_{i\in I}a_{i}\right)
=\tsum\limits_{i\in I}\left( a\cdot a_{i}\right) $ and $\left(
\tsum\limits_{i\in I}a_{i}\right) \cdot a=\tsum\limits_{i\in I}\left(
a_{i}\cdot a\right) $ for any $a\in R$.
\end{definition}

One can easily check that the set $D$ of all unions of double cosets of $A$
in $G$ together with the empty set is a starred, idempotent, complete
semiring under the following: addition is the operation of set union,
multiplication is the setwise product, $0$ is the empty set (we define the
product of the empty set with any subset of $G$ to be the empty set), $1$ is
the trivial double coset $A1_{G}A=A$ and the star operation maps each
non-empty $S\in D$ to its inverse $S^{-1}:=\left\{ s^{-1}|s\in S\right\} $
and $\emptyset $ to $\emptyset $. We shall call this structure the \emph{DC
dioid} associated with $G$ and $A$. Note that the natural partial order
relation defined above for a general dioid amounts in the case of the DC
dioid to set inclusion. Also note that the star operation is involutive,
that it is an isomorphism of the additive substructure and an
anti-isomorphism of the multiplicative substructure, and that $aa^{\ast
}=1+\ldots $ for all $\emptyset \neq a\in D$. In Section \ref{SECT_BN}, when
computing in the DC dioid, we will use $\cup $, $\subseteq $ and $^{-1}$
for, respectively, $+$, $\leq $ and $^{\ast }$.

We set $d_{j}:=$ $AjA$, for any $j\in J$. Any element of the DC dioid can
then be written as an infinitary sum $\tsum\limits_{j\in J}c_{j}d_{j}$,
where for each $j\in J$, $c_{j}$ is either $0$ or $1$ ($0,1\in D$). In
particular $G=\tsum\limits_{j\in J}d_{j}$, and the product of any two double
cosets defines the structure constants $c_{xyj}\in \left\{ 0,1\right\}
\subseteq D$ via $d_{x}d_{y}=\tsum\limits_{j\in J}c_{xyj}d_{j}$. Hence, in
this framework, the existence of two double cosets of $G$ whose product
equals $G$ is equivalent to the existence of $x,y\in J$ such that $%
d_{x}d_{y}=G$, which is equivalent to $c_{xyj}=1$ for all $j\in J$.

Finally observe the easy connection between the two approaches (in case of $%
G $ finite): $c_{xyj}=1$ if and only if $a_{xyj}\neq 0$.

\section{Groups with a $BN$-pair\label{SECT_BN}}

Throughout this section $G$ is a fixed group with a fixed $BN$-pair and the
associated Weyl group $W$. Recall that $W$ is generated by a set of
involutions $\left\{ s_{i}|i\in I\right\} $. We work in the DC dioid defined
by the double cosets of $B$ in $G$ (see Section \ref{Sect_DC}). By
Proposition 2.1.2 of \cite{Carter2} there is a bijective map between the set
of all double cosets of $B$ in $G$ and the set of all elements of $W$. Thus
we can choose the set $J$ (see Section \ref{Sect_DC}) of distinct $B$-double
cosets representatives as a subset of $N$, in the form $J=\left\{ n_{w}\in
N|n_{w}H=w,w\in W\right\} $, where $H=B\cap N$ and $\left\vert J\right\vert
=\left\vert W\right\vert $. The specific choices of the $n_{w}$ do not
matter since the double cosets of $B$ are uniquely labeled by the Weyl group
elements. Hence, to simplify the notation, we write $d_{w}$ for the double
coset $d_{n_{w}}:=Bn_{w}B$. Observe that $d_{w}^{-1}=d_{w^{-1}}$ for every $%
w\in W$ since $\left( Bn_{w}B\right) ^{-1}=Bn_{w}^{-1}B=Bn_{w^{-1}}B$. The
following lemma is a basic tool for computing products in the DC dioid,
given these settings.

\begin{lemma}
\label{Lem_dwds}For all $w\in W$, and $i\in I$ we have:%
\begin{equation}
d_{w}d_{s_{i}}=\left\{ 
\begin{array}{c}
\text{ }d_{ws_{i}}\text{\ \ \ \ \ \ \ if }l\left( ws_{i}\right) =l\left(
w\right) +1 \\ 
d_{ws_{i}}\cup d_{w}\text{ if }l\left( ws_{i}\right) =l\left( w\right) -1%
\text{,}%
\end{array}%
\right.  \label{Eq_dwds}
\end{equation}%
\begin{equation}
d_{s_{i}}d_{w}=\left\{ 
\begin{array}{c}
d_{s_{i}w}\text{ \ \ \ \ \ \ \ if }l\left( s_{i}w\right) =l\left( w\right) +1
\\ 
d_{s_{i}w}\cup d_{w}\text{ if }l\left( s_{i}w\right) =l\left( w\right) -1%
\text{.}%
\end{array}%
\right.  \label{Eq_dsdw}
\end{equation}
\end{lemma}

\begin{proof}
Note that (\ref{Eq_dwds}) is obtained from (\ref{Eq_dsdw}) by taking the
inverse of both sides of (\ref{Eq_dsdw}), and vice versa. Hence it suffices
to prove (\ref{Eq_dsdw}). If $l\left( s_{i}w\right) =l\left( w\right) +1$
the claim follows from \cite{Carter2}, Proposition 2.1.3 (ii).

Now suppose that $l\left( s_{i}w\right) =l\left( w\right) -1$. First we
prove the claim for the case $w=s_{i}$. Denote $n_{s_{i}}$ by $n_{i}$. We
have $d_{s_{i}}^{2}=Bn_{i}Bn_{i}B$, and so we have to show that $%
Bn_{i}Bn_{i}B=B\cup Bn_{i}B$. The l.h.s is contained in the r.h.s, by axiom
(v) of the definition of a group with a $BN$-pair (see Section \ref%
{Sect_Intro}). To prove the reverse inclusion, observe that since $%
n_{i}^{2}\in B$ we get $B\subseteq Bn_{i}Bn_{i}B$. On the other hand, by the
definition of a group with a $BN$-pair, $n_{i}Bn_{i}\neq B$ so $%
Bn_{i}Bn_{i}B\neq B$. Since $Bn_{i}Bn_{i}B\subseteq B\cup Bn_{i}B$, this
implies $Bn_{i}B\subseteq Bn_{i}Bn_{i}B$, and altogether $%
Bn_{i}Bn_{i}B=B\cup Bn_{i}B$, which is equivalent to $%
d_{s_{i}}d_{s_{i}}=d_{1}\cup d_{s_{i}}$.

Now consider a general $w\in W$ of positive length. Using $s_{i}^{2}=1$ and $%
l\left( s_{i}w\right) =l\left( w\right) -1$, we get $l\left( w\right)
=l\left( s_{i}\left( s_{i}w\right) \right) =l\left( s_{i}w\right) +1$ and by
the upper branch of the identity, which is already accounted for, $%
d_{s_{i}}d_{s_{i}w}=d_{s_{i}^{2}w}=d_{w}$. Multiply the last relation by $%
d_{s_{i}}$ on the left. We get:%
\begin{equation*}
d_{s_{i}}d_{w}=\left( d_{s_{i}}\right) ^{2}d_{s_{i}w}=\left( d_{1}\cup
d_{s_{i}}\right) d_{s_{i}w}=d_{s_{i}w}\cup
d_{s_{i}}d_{s_{i}w}=d_{s_{i}w}\cup d_{w}\text{.}
\end{equation*}
\end{proof}

\bigskip

An expression for $w\in W$ as a product of $l\left( w\right) $ generators $%
s_{i}$ is called a reduced word for $w$. Since reduced words are, in
general, not unique, we introduce the following formalism in order to make
our arguments precise. Let $\mathcal{E}_{I}$ be the free monoid generated by
the letters $\left\{ \varepsilon _{i}|i\in I\right\} $. Let $\varepsilon
_{0} $ be the identity of $\mathcal{E}_{I}$, and define a length function $l:%
\mathcal{E}_{I}\rightarrow \mathbb{N}_{0}$ by $l\left( \varepsilon
_{0}\right) =0$, $l\left( \varepsilon _{i}\right) =1$ for all $i\in I$, and
recursively $l\left( \varepsilon _{i_{1}}\varepsilon _{i_{2}}\cdots
\varepsilon _{i_{m}}\right) =l\left( \varepsilon _{i_{1}}\varepsilon
_{i_{2}}\cdots \varepsilon _{i_{m-1}}\right) +l\left( \varepsilon
_{i_{m}}\right) $ for all $m\geq 2$, where $i_{1},\ldots ,i_{m}\in I\cup
\left\{ 0\right\} $ (assume $0\notin I$). We denote both length functions on 
$\mathcal{E}_{I}$ and on $W$ by the same letter $l$. Let $\tau :\mathcal{E}%
_{I}\rightarrow W$ be the monoid homomorphism defined by $\tau \left(
\varepsilon _{0}\right) =1_{W}$ and $\tau \left( \varepsilon _{i}\right)
=s_{i}$ for all $i\in I$. We will call $a=\varepsilon _{i_{1}}\varepsilon
_{i_{2}}\cdots \varepsilon _{i_{m}}\in \mathcal{E}_{I}$ reduced if $l\left(
a\right) =l\left( \tau \left( a\right) \right) $. For each $w\in W$ define:%
\begin{equation*}
Red\left( w\right) :=\left\{ a\in \tau ^{-1}\left( w\right) |a\text{ is
reduced}\right\} \text{,}
\end{equation*}%
namely, $Red\left( w\right) $ is the set of all reduced expressions for $w$.
Clearly $Red\left( w\right) \neq \emptyset $ for any $w\in W$. Finally,
recall that if $W$ is finite then $w_{0}$ denotes the unique element of
maximal length in $W$.

\begin{lemma}
\label{Lem_d_t(xy)=d_t(x)d_t(y)}Let $x,y\in\mathcal{E}_{I}$ be such that $x$%
, $y$ and $xy$ are reduced. Then $d_{\tau\left( xy\right) }=d_{\tau\left(
x\right) }d_{\tau\left( y\right) }$.
\end{lemma}

\begin{proof}
By induction on $l\left( y\right) $. For $l\left( y\right) =0$ we have $%
y=\varepsilon_{0}$ so $xy=x\varepsilon_{0}=x$ and $\tau\left( y\right)
=\tau\left( \varepsilon_{0}\right) =1_{W}$. In this case $d_{\tau\left(
xy\right) }=d_{\tau\left( x\right) }$ and $d_{\tau\left( x\right)
}d_{\tau\left( y\right) }=d_{\tau\left( x\right) }B=d_{\tau\left( x\right) }$
so the claim holds. If $l\left( y\right) =1$ then $y=\varepsilon_{i}$ with $%
i\in I$. We have $d_{\tau\left( x\right) }d_{\tau\left( y\right)
}=d_{\tau\left( x\right) }d_{s_{i}}$ and since $xy$ is reduced, $l\left(
xy\right) =$ $l\left( \tau\left( xy\right) \right) =l\left( \tau\left(
x\right) s_{i}\right) $. Since $x$ is reduced, $l\left( xy\right) =l\left(
x\varepsilon_{i}\right) =l\left( x\right) +1=l\left( \tau\left( x\right)
\right) +1$, and so we proved $l\left( \tau\left( x\right) s_{i}\right) =$ $%
l\left( \tau\left( x\right) \right) +1$. By Lemma \ref{Lem_dwds} we get $%
d_{\tau\left( x\right) }d_{s_{i}}=d_{\tau\left( x\right)
s_{i}}=d_{\tau\left( xy\right) }$.

Now suppose that $l\left( y\right) =m>1$. Thus $y=\varepsilon
_{i_{1}}\varepsilon _{i_{2}}\cdots \varepsilon _{i_{m}}$ for some $%
i_{1},\ldots ,i_{m}\in I$. Since $y$ is reduced, $l\left( y\right) =l\left(
\tau \left( y\right) \right) $, and hence $y=\varepsilon _{i_{1}}y^{\prime }$%
, and $l\left( y^{\prime }\right) =l\left( \tau \left( y^{\prime }\right)
\right) =m-1$, so $y^{\prime }\in \mathcal{E}_{I}-\left\{ \varepsilon
_{0}\right\} $ is reduced of length $l\left( y^{\prime }\right) =l\left(
y\right) -1$. By Lemma \ref{Lem_dwds} we have $d_{\tau \left( y\right)
}=d_{s_{i_{1}}}d_{\tau \left( y^{\prime }\right) }$. By assumption, $%
xy=x\varepsilon _{i_{1}}y^{\prime }$ is also reduced. Therefore $l\left(
\tau \left( x\varepsilon _{i_{1}}\right) \right) =l\left( x\right) +1$,
because, otherwise, $l\left( \tau \left( x\varepsilon _{i_{1}}\right)
\right) <l\left( x\right) +1$ and so $l\left( \tau \left( xy\right) \right)
=l\left( \tau \left( x\varepsilon _{i_{1}}\right) \tau \left( y^{\prime
}\right) \right) <l\left( x\right) +1+l\left( y\right) -1=l\left( x\right)
+l\left( y\right) =l\left( xy\right) $ contradicting the fact the $xy$ is
reduced. Now we can use the induction assumption with $x\varepsilon _{i_{1}}$
in the role of $x$ and $y^{\prime }$ in the role of $y$. We get $d_{\tau
\left( xy\right) }=d_{\tau \left( x\varepsilon _{i_{1}}y^{\prime }\right) }=$
$d_{\tau \left( x\varepsilon _{i_{1}}\right) }d_{\tau \left( y^{\prime
}\right) }$. Using twice the length $1$ case (it works for both orderings)
gives $d_{\tau \left( x\varepsilon _{i_{1}}\right) }d_{\tau \left( y^{\prime
}\right) }=d_{\tau \left( x\right) }d_{s_{i_{1}}}d_{\tau \left( y^{\prime
}\right) }=d_{\tau \left( x\right) }d_{\tau \left( \varepsilon
_{i_{1}}y^{\prime }\right) }=$ $d_{\tau \left( x\right) }d_{\tau \left(
y\right) }$.
\end{proof}

\begin{lemma}
\label{Lem_d_t(w0)ContainedInd_t(b)d_t(w0)} Assume that $W$ is finite. Let $%
b\in \mathcal{E}_{I}$. Then $d_{w_{0}}\subseteq d_{\tau \left( b\right)
}d_{w_{0}}$.
\end{lemma}

\begin{proof}
By induction on $l\left( \tau \left( b\right) \right) $. If $l\left( \tau
\left( b\right) \right) =0$ then $d_{\tau \left( b\right) }$ is the
multiplicative identity of the DC dioid and the claim is clear. If $l\left(
\tau \left( b\right) \right) >0$ choose a reduced expression $\varepsilon
_{i_{1}}\varepsilon _{i_{2}}\cdots \varepsilon _{i_{m}}\in Red\left(
b\right) $, and set $b^{\prime }=\varepsilon _{i_{1}}\varepsilon
_{i_{2}}\cdots \varepsilon _{i_{m-1}}$. By Lemma \ref%
{Lem_d_t(xy)=d_t(x)d_t(y)} we have $d_{\tau \left( b\right) }=d_{\tau \left(
b^{\prime }\right) }d_{\tau \left( \varepsilon _{i_{m}}\right) }=d_{\tau
\left( b^{\prime }\right) }d_{s_{i_{m}}}$. Hence $d_{\tau \left( b\right)
}d_{w_{0}}=d_{\tau \left( b^{\prime }\right) }d_{s_{i_{m}}}d_{w_{0}}$. By
maximality of $l\left( w_{0}\right) $, we have $l\left(
s_{i_{m}}w_{0}\right) <l\left( w_{0}\right) $, hence, by Lemma \ref{Lem_dwds}
(\ref{Eq_dsdw}), $d_{w_{0}}\subseteq d_{s_{i_{m}}}d_{w_{0}}$ and so $d_{\tau
\left( b^{\prime }\right) }d_{w_{0}}\subseteq d_{\tau \left( b^{\prime
}\right) }d_{s_{i_{m}}}d_{w_{0}}=d_{\tau \left( b\right) }d_{w_{0}}$. Since $%
l\left( \tau \left( b^{\prime }\right) \right) =m-1<m=l\left( \tau \left(
b\right) \right) $ we have by induction $d_{w_{0}}\subseteq d_{\tau \left(
b^{\prime }\right) }d_{w_{0}}$ and the claim follows.
\end{proof}

\begin{lemma}
\label{Lem_d_t(ab)ContainedInd_t(a)d_t(b)}Let $a,b\in\mathcal{E}_{I}$. Then $%
d_{\tau\left( ab\right) }\subseteq d_{\tau\left( a\right) }d_{\tau\left(
b\right) }$.
\end{lemma}

\begin{proof}
We have $d_{\tau \left( a\right) }d_{\tau \left( b\right) }=\left( Bn_{\tau
\left( a\right) }B\right) \left( Bn_{\tau \left( b\right) }B\right)
=Bn_{\tau \left( a\right) }Bn_{\tau \left( b\right) }B$. Since $1_{G}\in B$
we get $n_{\tau \left( a\right) }n_{\tau \left( b\right) }\in n_{\tau \left(
a\right) }Bn_{\tau \left( b\right) }$. Hence $n_{\tau \left( a\right)
}n_{\tau \left( b\right) }\in d_{\tau \left( a\right) }d_{\tau \left(
b\right) }$ and $Bn_{\tau \left( a\right) }n_{\tau \left( b\right)
}B\subseteq d_{\tau \left( a\right) }d_{\tau \left( b\right) }$. Now, $\tau $
is a homomorphism, so $\tau \left( ab\right) =\tau \left( a\right) \tau
\left( b\right) $, and hence $d_{\tau \left( ab\right) }=Bn_{\tau \left(
ab\right) }B=Bn_{\tau \left( a\right) \tau \left( b\right) }B$. Now $\left(
n_{\tau \left( a\right) }H\right) \left( n_{\tau \left( b\right) }H\right)
=n_{\tau \left( a\right) \tau \left( b\right) }H$, and since $H\subseteq B$
this implies $Bn_{\tau \left( a\right) \tau \left( b\right) }B=Bn_{\tau
\left( a\right) }n_{\tau \left( b\right) }B$. Thus $d_{\tau \left( ab\right)
}=Bn_{\tau \left( a\right) }n_{\tau \left( b\right) }B$ and the claim
follows.
\end{proof}

Note that one can also prove Lemma \ref{Lem_d_t(w0)ContainedInd_t(b)d_t(w0)}%
, using Lemma \ref{Lem_d_t(xy)=d_t(x)d_t(y)} and Lemma \ref{Lem_dwds} (\ref%
{Eq_dsdw}), without explicitly invoking the definition of the $d_{w}$.

\begin{proof}[Proof of Theorem \protect\ref{Th_G=BB_B for BN-pair}]
Let $w\in W$ be arbitrary. We will show that $d_{w}\subseteq d_{w_{0}}^{2}$.
If $w=1_{W}$ we get $d_{w}=B$ and $d_{w_{0}}^{2}=BB^{n_{0}}B$ since $%
w_{0}^{2}=1$. Since $B\subseteq BB^{n_{0}}B$ our claim holds for this case
and hence we can assume that $l\left( w\right) >0$.

Fix $a\in Red\left( ww_{0}\right) $. Since $w\neq1$, we have $ww_{0}\neq
w_{0}$ and hence $l\left( ww_{0}\right) <l\left( w_{0}\right) $. Hence there
exists $b\in\mathcal{E}_{I}$ such that $b$ is reduced, $l\left( b\right) >0$
and $ab\in Red\left( w_{0}\right) $ (see the first paragraph of the proof of
Proposition 2.5.5 of \cite{Carter2} and note that it relies on properties of 
$W$ that apply for the general case of a group with a $BN$-pair and a finite
Weyl group $W$). We have $\tau\left( a\right) =ww_{0}$ and hence $\tau\left(
ab\right) =\tau\left( a\right) \tau\left( b\right) =ww_{0}\tau\left(
b\right) $. On the other hand $ab\in Red\left( w_{0}\right) $ implies that $%
\tau\left( ab\right) =w_{0}$, so, substituting this in the previous
relation, we obtain $w_{0}=ww_{0}\tau\left( b\right) $ and hence $\tau\left(
b\right) =w_{0}w^{-1}w_{0}$. Furthermore, $%
d_{w_{0}}^{2}=d_{w_{0}}d_{w_{0}}=d_{\tau\left( ab\right) }d_{w_{0}}$. By
Lemma \ref{Lem_d_t(xy)=d_t(x)d_t(y)}, $d_{\tau\left( ab\right)
}=d_{\tau\left( a\right) }d_{\tau\left( b\right) }$ and hence we proved $%
d_{w_{0}}^{2}=d_{\tau\left( a\right) }d_{\tau\left( b\right) }d_{w_{0}}$.

By Lemma \ref{Lem_d_t(w0)ContainedInd_t(b)d_t(w0)} we have $%
d_{w_{0}}\subseteq d_{\tau\left( b\right) }d_{w_{0}}$. Hence $d_{\tau\left(
a\right) }d_{w_{0}}\subseteq d_{\tau\left( a\right) }d_{\tau\left( b\right)
}d_{w_{0}}=d_{w_{0}}^{2}$. Now we get from Lemma \ref%
{Lem_d_t(ab)ContainedInd_t(a)d_t(b)} (taking, in that lemma, $b\in\mathcal{E}%
_{k}$ such that $\tau\left( b\right) =w_{0}$) $d_{\tau \left( a\right)
w_{0}}\subseteq d_{\tau\left( a\right) }d_{w_{0}}$. However, $a\in Red\left(
ww_{0}\right) $, so $\tau\left( a\right) =ww_{0}$, giving $w=\tau\left(
a\right) w_{0}$, and so $d_{w}\subseteq d_{\tau\left( a\right) }d_{w_{0}}$.
Combining this with $d_{\tau\left( a\right) }d_{w_{0}}\subseteq
d_{w_{0}}^{2} $ we obtain $d_{w}\subseteq d_{w_{0}}^{2}$ as desired.
\end{proof}

\begin{proof}[Proof of Remark \protect\ref{Rem_InfiniteWeyl}]
An arbitrary product of double cosets of $B$ in $G$ takes the form, in the
associated DC dioid, $d_{w_{1}}\cdots d_{w_{m}}$ for some $w_{1},\ldots
,w_{m}\in W$. For each $1\leq j\leq m$, choose a reduced expression $%
w_{j}=s_{j_{1}}\cdots s_{j_{l\left( w_{j}\right) }}$. Applying Lemma \ref%
{Lem_d_t(xy)=d_t(x)d_t(y)}, we get $d_{w_{j}}=d_{s_{j_{1}}}\cdots
d_{s_{j_{l\left( w_{j}\right) }}}$. Thus, any product of double cosets of $B$
in $G$, is a finite length product of some $d_{s_{i}}$ for some $i\in I$
(with possible repetitions). Now, employing Lemma \ref{Lem_dwds}, it follows
by induction on the number of the $d_{s_{i}}$ factors in the product, that
such a product is a union of a finite number of various $d_{w}$ with $w\in W$%
. Since $G$ is equal to the union of the infinite set $\left\{ d_{w}|w\in
W\right\} $, no product of double cosets of $B$ in $G$ is equal to $G$.
\end{proof}

\begin{proof}[Proof of Corollary \protect\ref{Coro_Gauss}]
The claim follows immediately from Theorem \ref{Th_G=BB_B for BN-pair} upon
substituting $B=UH$ in $G=BB^{n_{0}}B$ and noting that $H$ commutes with
both $U$ and $n_{0}$.
\end{proof}

\section{The topological approach\label{SECT_TOPO}}

This section illustrates the topological approach to products of double
cosets.

\begin{proof}[Proof of Proposition \protect\ref{Prop_semiTopo}]
Let $g\in G$. We have to prove that $g\in AA^{-1}$. Consider $gA$. This is
an open subset of $G$, and hence, since $A$ is dense in $G$, it intersects $%
A $, that is $gA\cap A\neq\emptyset$. It follows that there exist $%
a,a^{\prime}\in A$ such that $ga^{\prime}=a$. This is equivalent to $%
g=a\left( a^{\prime}\right) ^{-1}\in AA^{-1}$.
\end{proof}

\begin{proof}[Proof of Corollary \protect\ref{Coro_GReducedG=BBBFromTopo}]
Observe that a linear algebraic group is a semi-topological group and hence
we can apply Proposition \ref{Prop_semiTopo} with $A=Bn_{0}B$, provided that
we show that $\left( Bn_{0}B\right) ^{-1}=Bn_{0}B$ and that $Bn_{0}B$ is an
open dense subset of $G$. The first claim is an immediate consequence of $%
w_{0}^{2}=1$ and the second is well known (see \cite{MalleTesterman}, proof
of Theorem 11.20).
\end{proof}

\section{$\protect\gamma_{\text{cp}}\left( G\right) =3$ for a non-solvable
finite $G$\label{Sect_gamma=3}}

In this section we prove Theorem \ref{Th_gamma=3}. For any group $G$ we
define $\beta _{\text{cp}}\left( G\right) $ as the smallest natural number $%
k $ such that there exist $A<G$ and $x\in G$ for which $G=\left( AxA\right)
^{k}$ ($\Longleftrightarrow G=AA^{x}\cdots A^{x^{k}}$), and $\beta _{\text{cp%
}}\left( G\right) =\infty $ if no such $k$ exists. Then $\beta _{\text{cp}%
}\left( G\right) $ satisfies the lifting property, namely, for every $%
N\trianglelefteq G$ we have $\beta _{\text{cp}}\left( G\right) \leq \beta _{%
\text{cp}}\left( G/N\right) $ (see Proposition 8 of \cite{GaronziLevy} and
note that $\left( xN\right) ^{i}=x^{i}N$). Assume, henceforth, that $G$ is a
finite non-solvable group. Following the proof of Theorem 3 of \cite%
{GaronziLevy}, using the same arguments, we reduce to the case where $G$ has
a unique minimal normal subgroup $N$, and $N$ is non-abelian. Thus we can
assume the minimal non-solvable setting (see the beginning of Section 3 of 
\cite{GaronziLevy}). Now one checks that the analysis carried in Section 3
of \cite{GaronziLevy} goes through also for $\beta _{\text{cp}}\left(
G\right) $ since the restriction that the conjugating elements are
successive powers of a given $x$ lifts from subgroups to direct products of
subgroups and to normalizers of subgroups. Thus, relying on an analogue of
Lemma 14 of \cite{GaronziLevy} for $\beta _{\text{cp}}\left( G\right) $, we
have the following reduction of the statement $\beta _{\text{cp}}\left(
G\right) =2$ for every finite non-solvable group $G$ to almost simple groups.

\begin{claim}
\label{Claim_BetaReduction}$\beta_{\text{cp}}\left( G\right) =2$ for every
finite non-solvable group $G$ if for any almost simple group $X$ with $%
S=soc\left( X\right) $, there exist $A<S$ and $x\in S$ such that:

(a) $S=\left( AxA\right) ^{2}$ $\ $and (b) $N_{X}\left( A\right) S=X$.
\end{claim}

\begin{proof}
By the discussion preceding the claim we can assume that $G$ satisfies the
minimal non-solvable setting of Section 3 of \cite{GaronziLevy}. Hence it is
sufficient to show that the existence of $A<S$ with the properties specified
in the claim, ensures the existence of $U\leq X$ satisfying all of the
assumptions of (the $\beta _{\text{cp}}\left( G\right) $ analogue of) Lemma
14 of \cite{GaronziLevy}. Indeed, we set $U:=N_{X}\left( A\right) $. Then $%
US=X$ is just (b). In order to prove that $S=\left( \left( U\cap S\right)
x\left( U\cap S\right) \right) ^{2}$ we note that $U\cap S=N_{S}\left(
A\right) \geq A$. Using this, the fact that for any $s\in S$ we have $\left(
N_{S}\left( A\right) \right) ^{s}=N_{S}\left( A^{s}\right) $, and Theorem %
\ref{Th_ConjDCEquiv}(2)(ii) we get that $S=\left( AxA\right)
^{2}=AA^{x}A^{x^{2}}$ implies $S=N_{S}\left( A\right) \left( N_{S}\left(
A\right) \right) ^{x}\left( N_{S}\left( A\right) \right) ^{x^{2}}=\left(
\left( N_{S}\left( A\right) \right) x\left( N_{S}\left( A\right) \right)
\right) ^{2}=\left( \left( U\cap S\right) x\left( U\cap S\right) \right)
^{2} $. Finally, $N_{S}\left( A\right) \geq A>1$ and $N_{S}\left( A\right)
<S $ since $S$ is simple and $A$ is proper in $S$ and non-trivial.
\end{proof}

Thus it remains to use the classification of simple non-abelian groups in
order to find an appropriate choice of $A$ for any almost simple group $X$.
As a matter of fact, in all cases we find $A$ $<S$, for each simple
non-abelian group $S$, such that this $A$ satisfies (a) and (b) for all
almost simple $X$ with $S=soc\left( X\right) $.

1. $S\cong Alt\left( n\right) $, the alternating group of degree $n$, and $%
n\geq5$. Here we use Corollary 15 of \cite{GaronziLevy}. Note that in all
cases there, $A<S$ is a point stabilizer of a $2$-transitive action of $S$
on some set. In this case $S=A\cup AxA$ for any $x\in S-A$. Furthermore, one
can choose $x$ to be an involution so clearly $A\subseteq\left( AxA\right)
^{2}$. On the other hand $\left( AxA\right) ^{2}\neq A$ since $%
A^{x}\subseteq\left( AxA\right) ^{2}$ and $A\neq A^{x}$ since $A$ is self
normalizing. Thus $S=\left( AxA\right) ^{2}$ so condition (a) of Claim \ref%
{Claim_BetaReduction} holds, while condition (b) is checked in the proof of
Corollary 15 of \cite{GaronziLevy}.

2. $S$ is a simple group of Lie type. Here we use the fact that all finite
simple groups of Lie type are groups with a $BN$-pair (See \cite{GLS}
Definition 2.2.8, Theorem 2.2.10 and Theorem 2.3.4. The Tits group is not
counted in this category). Hence, by Theorem \ref{Th_G=BB_B for BN-pair}, $%
S=\left( Bn_{0}B\right) ^{2}$ where $B$ is a Borel subgroup of $S$ and $%
n_{0}\in S$ is any preimage of the longest element of the Weyl group. Thus
we take $A=B$ in Claim \ref{Claim_BetaReduction}, and this choice satisfies
condition (a). It is clear that $B<S$ (see axiom (iv) for groups with a $BN$%
-pair). It remains to check that condition (b) of Claim \ref%
{Claim_BetaReduction} is satisfied, that is, that $N_{X}\left( B\right) S=X$%
. Let $p$ be the defining characteristic of $S$. Then, by \cite{GLS},
Theorem 2.6.5 (d) (taking there $J=\emptyset$) we have $B=N_{S}\left(
P\right) $ where $P$ is a Sylow $p$-subgroup of $S$. Let $x\in X$. Then%
\begin{equation*}
B^{x}=N_{S}\left( P\right) ^{x}=N_{S}\left( P^{x}\right) \text{.}
\end{equation*}
But $P^{x}$ is a Sylow $p$-subgroup of $S$ so by Sylow's theorem there exist 
$s\in S$ such that $P^{x}=P^{s}$ for some $s\in S$. Hence $B^{x}=N_{S}\left(
P^{s}\right) =N_{S}\left( P\right) ^{s}=B^{s}$. Thus we have proved that for
each $x\in X$ there exists $s\in S$ such that $B^{x}=B^{s}$. The last
equality is equivalent to $B^{xs^{-1}}=B$ which is equivalent to $xs^{-1}\in
N_{X}\left( B\right) $. From here we get $x\in N_{X}\left( B\right)
s\subseteq N_{X}\left( B\right) S$ and $N_{X}\left( B\right) S=X$ follows.

3. $S$ is one of the $26$ sporadic simple groups or $S$ is the Tits group $%
^{2}F_{4}\left( 2\right) ^{\prime }$. First note that for all of the $27$
groups under discussion $\left\vert Aut\left( S\right) :S\right\vert \in
\left\{ 1,2\right\} $. Therefore, for a fixed $S$, we have two
possibilities: (i) $Aut\left( S\right) =S$ and then it suffices to find $A<S$
satisfying (a) of Claim \ref{Claim_BetaReduction} or (ii) $\left\vert
Aut\left( S\right) :S\right\vert =2$ in which case it suffices to find a
maximal $A<S$ satisfying (a) of Claim \ref{Claim_BetaReduction} and $%
N_{Aut\left( S\right) }\left( A\right) \neq A$. Note that $N_{Aut\left(
S\right) }\left( A\right) \neq A$ for $A$ maximal in $S$ ensures that $%
N_{Aut\left( S\right) }\left( A\right) $ contains an element of $Aut\left(
S\right) -S$ and hence $N_{Aut\left( S\right) }\left( A\right) S=Aut\left(
S\right) $. Information about maximal subgroups $A$ of the simple groups $S$
in question, satisfying $N_{Aut\left( S\right) }\left( A\right) \neq A$ is
readily available in the ATLAS \cite{ATLAS}. The verification of condition
(a) for the candidate subgroup was done using GAP \cite{GAP} and MAGMA \cite%
{MAGMA}. For twenty six out of the twenty seven simple groups under
discussion, we have used the GAP package \texttt{mfer} (\cite%
{GAPCTblLib1.2.1}, \cite{BreuerMuller}). This package, written by T. Breuer,
I. H\"{o}hler and J. M\"{u}ller, contains a database which enables one to
compute the collapsed adjacency matrices associated with various
multiplicity-free permutation modules arising from actions of the simple
sporadic groups. Given these matrices we can use the method explained in the
second part of Section \ref{Sect_DC} in order to look for an appropriate
double coset of $A$ where $A$ is a point stabilizer of the action. More
precisely, let $A<S$ be a subgroup such that the permutation character $%
1_{A}^{S}$ is multiplicity-free, and let $P_{1},\ldots ,P_{r}$ be the $r$
collapsed adjacency matrices computed by \texttt{mfer}. We had run a simple
GAP program which loops over $1\leq i\leq r$, and for each $i$ checks
whether all of the entries of the $i$-th row of $P_{i}$ are non-zero. As
explained in Section \ref{Sect_DC}, the double coset labeled by $i$ squares
to $G$ if and only if $P_{i}$ satisfies this condition. Note that the 
\texttt{mfer} package also allows the computation of collapsed adjacency
matrices associated with multiplicity-free permutation modules of groups
which appear in GAP's table of marks library (TOM, see \cite%
{NaughtonPfeiffer}). We have used this feature for the Tits group $%
{}^{2}F_{4}\left( 2\right) ^{\prime }$.

For $S=O^{\prime }N$ we have $\left\vert Aut\left( S\right) :S\right\vert =2$%
, and the \texttt{mfer} database does not contain a subgroup $A$ satisfying $%
N_{Aut\left( S\right) }\left( A\right) \neq A$. Thus, for $S=O^{\prime }N$,
we cannot use the collapsed adjacency matrices method. Instead, we have
verified condition (a) of Claim \ref{Claim_BetaReduction} for $A=J_{1}$
using a different algorithm, which we implemented in MAGMA. Note that $%
N_{Aut\left( S\right) }\left( A\right) \neq A$ (see \cite{ATLAS}). This
algorithm relies on the possibility to construct representatives $%
x_{1},\ldots ,x_{r}$ of the $r$ distinct $A$ double cosets (a built-in MAGMA
function) and on the ability to test, for each $1\leq i\leq r$, if a given $%
s\in S$ belongs to the $Ax_{i}A$ (without computing the full set $Ax_{i}A$).
This function is not provided by MAGMA and had to be written separately. The
main steps of the algorithm are:

1. Calculate a set $\left\{ x_{1}=1,x_{2},\ldots ,x_{r}\right\} $ of
distinct $A$-double cosets representatives.

2. For each $2\leq i\leq r$ check if $x_{i}^{-1}\in Ax_{i}A$ - this
guarantees $A\subseteq\left( Ax_{i}A\right) ^{2}$ which is necessary for $%
S=\left( Ax_{i}A\right) ^{2}$.

3. For each double coset $AxA$, $x\in \left\{ x_{2},\ldots ,x_{r}\right\} $,
satisfying the necessary condition in 2, check whether $\left( AxA\right)
^{2}=S$ by checking $\left( AxA\right) ^{2}\cap \left( Ax_{j}A\right) \neq
\emptyset $ for all $2\leq j\leq r$ using the following probabilistic
function (TRIALS is a predefined positive integer constant): Choose $a\in A$
at random TRIALS times. For each $a$ find the unique $2\leq j\leq r$ such
that $xax\in Ax_{j}A$ and mark it (clearly $\left( AxA\right) ^{2}\cap
\left( Ax_{j}A\right) \neq \emptyset $ if and only if $xax\in Ax_{j}A$ for
some $a\in A$). If all $j\in \left\{ 2,\ldots ,r\right\} $ are marked after
TRIALS trials, then $\left( AxA\right) ^{2}=S$ with certainty. Otherwise, we
can only estimate the probability that $\left( AxA\right) ^{2}\neq S$, but
this is not needed for our purpose. One should choose TRIALS big enough
compared to $r$, taking into consideration the relative sizes of the
non-trivial double cosets.

Table \ref{TblSporadics} in the appendix summarizes the results of the
computations described above by listing pairs $\left( S,A\right) $ such that 
$S$ varies over all of the $26$ sporadic simple groups and the Tits group.
For each $S$ it displays one choice of $A<S$ that satisfies Claim \ref%
{Claim_BetaReduction}.

\begin{proof}[Proof of Corollary \protect\ref{Coro_AtMost3}]
If $G$ is non-solvable then Theorem \ref{Th_ConjDCEquiv} and Theorem \ref%
{Th_gamma=3} imply the existence of $A<G$ and $x\in G$ such that $G=AA^{x}A$%
. Suppose now that $G$ is a finite solvable group which is not a cyclic $p$%
-group. Then $G=NH=NHN$ where $N$ is a maximal normal subgroup and $H$ is
any proper subgroup which is not contained in $N$. For if $H<G$, $H\not\leq
N $ then $HN/N\leq G/N$ is non-trivial. But $G/N$ is cyclic of order $p$ by
the solvability of $G$ so $HN/N=G/N$ and $G=HN$. If there is no $H<G$, $%
H\not\leq N$, then any $g\in G-N$ generates $G$ and so $G$ is cyclic. Since $%
G/N$ is of order $p$, we can choose $g$ to be a $p$-element and hence $G$ is
a cyclic $p$-group - a contradiction.
\end{proof}

\section{Finite Coxeter Groups and products of conjugate parabolics\label%
{SECT_Coxeter}}

Recall that the Weyl group of a finite simple group of Lie type is in
particular a finite Coxeter group. A finite Coxeter group $C$ is a finite
group with presentation $\left\langle s_{1},\ldots ,s_{n}|\left(
s_{i}s_{j}\right) ^{m_{ij}}=1,\forall 1\leq i,j\leq n\right\rangle $ where
the positive integers $m_{ij}$ satisfy: (i) $m_{ij}=m_{ji}$, $\forall 1\leq
i,j\leq n$ (ii) $m_{ii}=1$ for all $1\leq i\leq n$ (equivalently, each $%
s_{i} $ is an involution) (iii) For any $1\leq i\neq j\leq n$, $m_{ij}\geq 2$
(note that $m_{ij}=2$ if and only if $s_{i}$ and $s_{j}$ commute). The
values of the $m_{ij}$ can be encoded in a finite, undirected, simple,
edge-labeled graph (the Coxeter graph) on $n$ vertices. The set of vertices
is $\left\{ s_{1},\ldots ,s_{n}\right\} $. For any $1\leq i\neq j\leq n$, $%
s_{i}$ and $s_{j}$ are connected by an edge if and only if $m_{ij}\geq 3$.
If $m_{ij}=3$ the edge is left unlabeled, and otherwise it is labeled by $%
m_{ij}$ ($\geq 4$). If the Coxeter graph is connected then the corresponding
Coxeter group is irreducible (i.e., it is not the direct product of two
Coxeter groups). The irreducible finite Coxeter groups were classified by
Coxeter (see \cite{Humphreys}, Section 2.4, Fig. 2.1 which lists all
connected Coxeter graphs). If $C$ is a finite Coxeter group with a Coxeter
graph $\Gamma $ then the Coxeter graph of a standard parabolic subgroup of $%
C $ is obtained from $\Gamma $ by deleting the vertices (and the edges
adjacent to them) corresponding to the $s_{i}$ that do not belong to the
subgroup. The maximal standard parabolic subgroups of $C$ are obtained by
deleting just a single $s_{i}$. It is clear that for the purpose of proving
Theorem \ref{Th_WeylParabolics}, it is sufficient to consider only the
maximal standard parabolic subgroups of $C$.

\begin{proof}[Proof of Theorem \protect\ref{Th_WeylParabolics}]
1. $C$ is of type $A_{n}$, $n\geq 2$. Then $C\cong S_{n+1}$, is a product of
three conjugates of a parabolic subgroup which is isomorphic to $S_{n}$ -
see Section \ref{Sect_gamma=3}, the discussion following the proof of Claim %
\ref{Claim_BetaReduction}. Note that the $2$-transitivity argument applies
for $S_{n+1}$ with $n\geq 2$.

2. $C$ is of type $B_{n}$, $n\geq 3$. In this case $C=V\rtimes S_{n}$ where%
\begin{equation*}
V=\left\{ \left( v_{1},\ldots ,v_{n}\right) |v_{i}\in \left\{ 0,1\right\} 
\text{, }1\leq i\leq n\right\} \cong Z_{2}^{n}
\end{equation*}%
is an elementary abelian $2$-group of rank $n$. We realize $S_{n}$ as $%
S_{\Omega }$ acting naturally (on the right) on the set $\Omega :=\left\{
1,\ldots ,n\right\} $. Then, for any $\pi \in S_{\Omega }$ and any $v=\left(
v_{1},\ldots ,v_{n}\right) \in V$ we have $\pi ^{-1}v\pi =\left( v_{\left(
1\right) \pi },\ldots ,v_{\left( n\right) \pi }\right) $. For each $1\leq
i\leq n$ set $C_{i}:=V_{i}\rtimes S_{\Omega _{i}}$, where $\Omega
_{i}:=\Omega -\left\{ i\right\} $ and $V_{i}$ is the subgroup of $V$
consisting of all $n$-tuples $\left( v_{1},\ldots ,v_{n}\right) $ satisfying 
$v_{i}=0$. Note that $C_{i}$ is a maximal parabolic subgroup of $C$.
Moreover, it is straightforward to check that if $g\in S_{\Omega }$
satisfies $\left( 1\right) g^{-1}=i$ then $C_{1}^{g}=C_{i}$. Fix $i\in
\left\{ 2,\ldots ,n\right\} $. We prove that $C=C_{1}C_{i}C_{1}$. We have $%
C_{1}C_{i}C_{1}=V_{1}S_{\Omega _{1}}V_{i}S_{\Omega _{i}}V_{1}S_{\Omega _{1}}$%
. Let $\pi \in S_{\Omega _{1}}$ be arbitrary. Then $V_{1}\pi
V_{i}=V_{1}V_{i}^{\pi ^{-1}}\pi $. Since $1$ is fixed by $\pi $, we have $%
V_{1}\neq V_{i}^{\pi ^{-1}}$. Hence $V_{1}$,$V_{i}^{\pi ^{-1}}$ are two
distinct subgroups of $V$ (recall that $V\trianglelefteq C$) of index $2$ in 
$V$, hence $\left\vert V_{1}V_{i}^{\pi ^{-1}}\right\vert \geq \frac{%
\left\vert V\right\vert ^{2}}{4\frac{\left\vert V\right\vert }{4}}%
=\left\vert V\right\vert $. Therefore $V_{1}V_{i}^{\pi ^{-1}}=V$. Since this
holds for any $\pi \in S_{\Omega _{1}}$ we can conclude that $V_{1}S_{\Omega
_{1}}V_{i}S_{\Omega _{i}}V_{1}S_{\Omega _{1}}=VS_{\Omega _{1}}S_{\Omega
_{i}}S_{\Omega _{1}}$. Finally, $S_{\Omega _{1}}S_{\Omega _{i}}S_{\Omega
_{1}}=S_{\Omega }$, by the same argument used for the case $C\cong S_{n+1}$
above.

3. $C$ is of type $D_{n}$, $n\geq 4$. In this case $C=U\rtimes S_{n}$ where $%
U$ is the subgroup of $V$ consisting of all $n$-tuples $\left( v_{1},\ldots
,v_{n}\right) $ satisfying $\tsum\limits_{i=1}^{n}v_{i}\equiv 0\left( \func{%
mod}2\right) $. All the arguments of the previous section carry through if
we replace $V_{i}$ by $U_{i}:=U\cap V_{i}$. Note that the parabolic subgroup 
$C_{1}$ obtained in this way is isomorphic to the Coxeter group of type $%
D_{n-1}$ (if $n=4$ set $D_{3}:=A_{3}$).

4. $C$ is of type $I_{2}\left( m\right) $, that is, $C$ is a dihedral group $%
Dih_{2m}$, $m\geq 3$ an integer (see \cite{Humphreys}, Section 2.4, Fig.
2.1). Note that $I_{2}\left( 3\right) =A_{2}$ is dealt with in part 1. For $%
C $ of type $I_{2}\left( 4\right) =B_{2}$ we get $\left\vert C\right\vert =8$
so $C$ is nilpotent and $\gamma _{\text{cp}}\left( C\right) =\infty $.
Finally, consider $C\cong Dih_{2m}$, $m\geq 5$. A maximal parabolic subgroup
is generated by a single involution and hence it is of order $2$. Since $%
\left\vert C\right\vert \geq 10$, we get that $C$ is not a product of three
conjugates of a maximal parabolic subgroup.

5. We have checked the remaining Coxeter groups by implementing criterion
2(i) of Theorem \ref{Th_ConjDCEquiv} in GAP. The functions provided by GAP
allow one to compute faithful permutation representations for each $C$ in
question, with explicit representations of the $s_{i}$. Thus it is possible
to compute the maximal parabolic subgroups, and, for each maximal parabolic
subgroup to compute its coset space and the orbits of its right
multiplication action on its own right cosets.
\end{proof}

\begin{proof}[Proof of Remark \protect\ref{Rem_FromWToG}]
By part 1 of Theorem \ref{Th_ConjDCEquiv}, $W=\left(
W_{I^{\prime}}zW_{I^{\prime}}\right) \left(
W_{I^{\prime}}wW_{I^{\prime}}\right) $ for some $z,w\in W$, where elements
of $W$ stand for left cosets of $H$. Now, $G=BWB$ (\cite{Carter2},
Proposition 2.1.1). Hence, substituting the expression for $W$ gives $%
G=B\left( W_{I^{\prime}}n_{z}W_{I^{\prime}}\right) \left(
W_{I^{\prime}}n_{w}W_{I^{\prime}}\right) B$, where $n_{z}$ and $n_{w}$ are
some fixed preimages in $N$ of respectively $z$ and $w$. In particular we
have 
\begin{equation*}
G\subseteq\left( \left( BW_{I^{\prime}}B\right) n_{z}\left( BW_{I^{\prime
}}B\right) \right) \left( \left( BW_{I^{\prime}}B\right) n_{w}\left(
BW_{I^{\prime}}B\right) \right) \text{.}
\end{equation*}
Applying part 1 of Theorem \ref{Th_ConjDCEquiv} again, it follows that $G$
is the product of three parabolic subgroups conjugate to $%
P_{I^{\prime}}=BW_{I^{\prime}}B$.
\end{proof}

\section*{{\protect\huge Appendix} \label{appendix}}

Table \ref{TblSporadics} lists, for each sporadic simple group $S$ including
the Tits group ${}^{2}F_{4}\left( 2\right) ^{\prime }$, one choice of a
subgroup $A<S$ satisfying the conditions of Claim \ref{Claim_BetaReduction}.
Note that the subgroups are specified by their isomorphism type, and there
may be more than one conjugacy class of subgroups belonging to the same
isomorphism type.

\begin{center}
\bigskip 
\begin{table}[h] \centering%
\begin{tabular}{|c|c||c|c|}
\hline
$S$ & $A$ & $S$ & $A$ \\ \hline
$M_{11}$ & $A_{6}.2_{3}$ & $O^{\prime }N$ & $J_{1}$ \\ \hline
$M_{12}$ & $M_{11}$ & $Co_{3}$ & $M^{c}L.2$ \\ \hline
$J_{1}$ & $L_{2}(11)$ & $Co_{2}$ & $U_{6}(2).2$ \\ \hline
$M_{22}$ & $L_{3}(4)$ & $Fi_{22}$ & $2.U_{6}(2)$ \\ \hline
$J_{2}$ & $U_{3}(3)$ & $HN$ & $A_{12}$ \\ \hline
$M_{23}$ & $M_{22}$ & $Ly$ & $G_{2}\left( 5\right) $ \\ \hline
${}^{2}F_{4}\left( 2\right) ^{\prime }$ & $L_{2}(25)$ & $Th$ & $%
^{3}D_{4}(2).3$ \\ \hline
$HS$ & $M_{22}$ & $Fi_{23}$ & $2.Fi_{22}$ \\ \hline
$J_{3}$ & $L_{2}(16).2$ & $Co_{1}$ & $Co_{2}$ \\ \hline
$M_{24}$ & $M_{23}$ & $J_{4}$ & $2^{11}:M_{24}$ \\ \hline
$M^{c}L$ & $U_{4}\left( 3\right) $ & $Fi_{24}^{\prime }$ & $Fi_{23}$ \\ 
\hline
$He$ & $S_{4}(4).2$ & $B$ & $2.^{2}E_{6}(2).2$ \\ \hline
$Ru$ & ${}^{2}F_{4}\left( 2\right) $ & $M$ & $2.B$ \\ \hline
$Suz$ & $G_{2}(4)$ &  &  \\ \hline
\end{tabular}
\caption
{Subgroups of simple sporadic groups and the Tits group with a double coset which squares to the group}%
\label{TblSporadics}%
\end{table}%
\end{center}

\end{document}